\documentclass{article}

\usepackage[utf8]{inputenc} 
\usepackage[T1]{fontenc} 
\usepackage[margin=1in]{geometry}
\usepackage{hyperref}       
\usepackage{url}            
\usepackage{booktabs}       
\usepackage{amsfonts}       
\usepackage{nicefrac}       
\usepackage{microtype}      
\usepackage{lipsum}
\usepackage{graphicx}
\graphicspath{ {./images/} }

\usepackage{natbib}
\usepackage{amsmath,amssymb,amsthm}

\usepackage{xcolor}
\usepackage{enumitem}
\usepackage{setspace}

\usepackage{color-edits}
\addauthor{dmt}{blue}

\allowdisplaybreaks

\newcommand{\Pb}{\mathbb{P}}

\newcommand{\R}{\mathbb{R}}

\newcommand{\E}{\mathbb{E}}

\newcommand{\F}{\mathcal{F}}

\newcommand{\lp}{\left(}
\newcommand{\rp}{\right)}
\newcommand{\lc}{\left\{}
\newcommand{\rc}{\right\}}
\newcommand{\lb}{\left[}
\newcommand{\rb}{\right]}
\newcommand{\rba}{\right|}
\newcommand{\lba}{\left|}
\newcommand{\rdba}{\right\|}
\newcommand{\ldba}{\left\|}
\newcommand{\ra}{\right\rangle}
\newcommand{\la}{\left\langle}

\newcommand{\tr}{\operatorname{tr}}
\newcommand{\bop}{\mathcal{B}(\hs)}
\newcommand{\sop}{\mathcal{B}^{\text{s}}(\hs)}
\newcommand{\pop}{\mathcal{B}^{\text{++}}(\hs)}
\newcommand{\hs}{\mathbb{H}}
\newcommand{\csop}{\mathcal{B}^{\text{cs}}(\hs)}
\newcommand{\csma}{\mathcal{B}^{\text{cs}}(\hs^d)}
\newcommand{\sma}{\mathcal{B}^{\text{s}}(\hs^d)}

\newtheorem{theorem}{Theorem}[section]

\newtheorem{lemma}[theorem]{Lemma}

\newtheorem{corollary}[theorem]{Corollary}

\newtheorem{fact}[theorem]{Fact}

\numberwithin{equation}{section}

\title{Intrinsic dimension concentration inequalities\\ for self-adjoint operators}

\author{
 Diego Martinez-Taboada$^{1}$ and Aaditya Ramdas$^{12}$ \\
  $^1$Department of Statistics \& Data Science\\
$^2$Machine Learning Department \\
  Carnegie Mellon University\\
  \texttt{\{diegomar,aramdas\}@andrew.cmu.edu} } 
  
\begin{document}
\maketitle
\begin{abstract}
We derive novel concentration inequalities for the operator norm of the sum of self-adjoint operators that do not explicitly depend on the underlying dimension of the operator, but rather an intrinsic notion of it. Our analysis leads to tighter results (in terms of constants) and simplified proofs. Our results unify the current intrinsic-dimension and ambient-dimension inequalities under independence, strictly improving both categories of bounds (such as by Tropp and Minsker). We present a general master theorem that we instantiate to obtain specific sub-Gaussian, Hoeffding, Bernstein, Bennett, and sub-exponential type inequalities. We also establish widely applicable concentration bounds under martingale dependence that provide tighter control than existing results.

\end{abstract}

\section{Introduction}

In high-dimensional statistics, machine learning, and quantum information theory, random matrices and operators are indispensable for modeling complex systems where individual interactions are unknown or too intricate for deterministic description. By representing data features or physical states as random operators, such as empirical covariance matrices or quantum density operators, we can characterize the structural properties of high-dimensional systems, including their orientation and spectral spread. A fundamental challenge in these fields is quantifying how the sum of independent random operators deviates from its expectation. This is addressed through operator-valued concentration inequalities, which provide probabilistic bounds on the fluctuations of spectral norms.

The focus is generally set on the operator norm deviation, because it serves as the fundamental metric for the worst-case stability and error of a linear system.  Essentially, controlling the operator norm allows for bounding the maximum possible distortion a random transformation can apply to any vector in the space. Unlike the Frobenius norm, the operator norm captures the spectral radius of the noise, providing the tightest possible control over the extreme eigenvalues of the system. 

The evolution of these inequalities has transitioned from scalar versions, such as the classical Hoeffding and Bernstein inequalities, to the groundbreaking work of \citet{ahlswede2002strong} and \citet{tropp2012user}, who extended these tools to the matrix setting. These results have become indispensable for analyzing everything from randomized numerical linear algebra to the recovery of signals in compressed sensing. 

\paragraph{The shift from ambient to intrinsic dimension.}
Historically, most matrix concentration inequalities have been dimension-dependent. For a sum of self-adjoint operators 
\begin{align} \label{eq:sn_definition}
    S_n = \sum_{i \leq n} X_i
\end{align}
acting on a $d$-dimensional space, the standard bounds typically feature a pre-factor of $d$. E.g., a common form for the $\sigma$-sub-Gaussian tail probability is $\Pb \lp \|S_n\| \geq r \rp \leq 2d \exp \lp-r^2/(2 \sigma^2) \rp.$ While powerful, these bounds suffer from two primary limitations:
\begin{itemize}
    \item \textbf{The infinite-dimensional barrier.} In functional analysis, kernel methods, and nonparametric statistics, we often work with operators on Hilbert spaces where $d = \infty$. In such cases, ambient-dimension bounds are vacuous, rendering them useless for theoretical guarantees in continuous domains.
    \item \textbf{Spectral non-isotropy.} In many practical applications, operators have rapidly decaying eigenvalues; they may be numerically low-rank even if their ambient dimension is large. Failing to capture this structure leads to overly pessimistic bounds that do not reflect the true concentrated nature of the data.
\end{itemize}

To address this, the concept of intrinsic dimension (or effective rank) was introduced. Intrinsic-dimension-dependent bounds replace the rigid $d$ with a quantity that reflects the actual ``spread'' of the operator's eigenvalues. This allows for sharper results in non-isotropic settings and provides a rigorous pathway to extend concentration results to infinite-dimensional operators.

\paragraph{Bridging the gap.}

Despite progress by \citet{minsker2017some} in developing intrinsic Bernstein bounds, a significant gap remains. Current inequalities  exist in two  regimes:
\begin{itemize}
    \item Intrinsic bounds that handle infinite dimensions, but suffer from suboptimal constants and complex proofs that do not simplify back to the ambient case cleanly.
    \item Ambient bounds that are sharper for isotropic operators but scale poorly with $d$.
\end{itemize}

This paper bridges this divide, offering a unified theory. We derive novel concentration inequalities for the operator norm of sums of self-adjoint operators that do not rely on the ambient dimension. We achieve tighter constants and a more streamlined mathematical framework. We strictly improve upon existing ambient-dimension inequalities while simultaneously providing the most refined intrinsic-dimension bounds to date. Specifically, we close the gap between these two methodologies by showing that our framework recovers and improves both as special cases. Beyond the theoretical unification, we provide a comprehensive suite of practical tools, including  sub-Gaussian, Hoeffding, Bernstein, Bennett, and sub-exponential-type inequalities. 

\section{Related work}

\paragraph{Operator-valued concentration inequalities.}   The results and proofs presented in this contribution improve on \citet{minsker2017some} and \citet{tropp2015introduction}, which provided operator-valued intrinsic-dimension concentration inequalities, with the former having introduced the key ideas of the analysis. Both works extended the earlier tail bounds of \citet{hsu2012tail}. Building on \citet{minsker2017some}, \citet{klochkov2020} obtained Bernstein-type inequalities for unbounded operators, and \citet{jirak2025concentration} derived bounds for heavy-tailed operators; the techniques developed here may therefore prove useful in further strengthening these extensions. Intrinsic-dimension results were also derived via PAC-Bayes arguments in \citet{zhivotovskiy2024dimension}. Other works that are dimension-free, but rank-dependent, are \citet{oliveira2010sums} and \citet{magen2011low}. 

\paragraph{Matrix-valued concentration inequalities.} The majority of the operator-valued bounds are dimension dependent, originating with the work of \citet{ahlswede2002strong}. A variety of sharper inequalities were later presented in \citet{tropp2012user} and \citet{tropp2015introduction}, with \citet{mackey2014matrix} obtaining a tighter Hoeffding-type inequality. Furthermore, \citet{tropp2011freedman} and \citet{wang2024sharp} considered matrix martingale sequences, which improved on \citet{oliveira2009concentration}. These dimension-dependent results were unified (and often strengthened) in \citet{howard2020time}. Related contributions include \citet{paulin2013deriving} and \citet{paulin2016efron}. Occasionally, specific results of these contributions extend to infinite dimensions, e.g., \citet[Theorem 3.3]{paulin2013deriving}, \citet[Theorem 7.1]{mackey2014matrix}, and \citet[Theorem 9.1]{paulin2016efron}. Our work strengthens some of these, replacing the ambient dimension by the intrinsic dimension.


\section{Background}

\paragraph{Self-adjoint compact operators.}
Let the set of bounded operators acting on a (real or complex) separable Hilbert space $\hs$ be denoted by $\bop$,  the set of self-adjoint operators by $\sop \subset \bop$,   the set of strictly positive operators by $\pop \subset \sop$, and the set of compact self-adjoint operators by $\csop \subset \sop$. For self-adjoint matrices where $\hs = \hs^d$ is a finite-dimension (real or complex) Hilbert space, it holds that $\sma = \csma$. We denote the spectrum of $A \in \csop$ by $\lambda(A)$, which is contained in $\R$, and its largest and smallest spectral values by $\lambda_{\max}(A)$ and $\lambda_{\min}(A)$. We further denote the operator norm by $\| A\|=-\lambda_{\min}(A) \vee \lambda_{\max}(A)$. We use $\tr(A)$ for trace-class operators; for positive operators, we also allow the extended trace in $[0,\infty]$. Compactness alone does not ensure a finite trace. A self-adjoint operator is Hilbert-Schmidt when $\|A\|_{\mathrm{HS}}^2=\tr(A^2)<\infty$. We employ the Loewner partial order $\preceq$, where $A \preceq B$ signifies that $B - A$ is positive semidefinite. The relations $A \preceq B$ imply $\tr(A) \leq \tr(B)$ whenever the traces are defined, and $\lambda_{\max}(A) \leq \lambda_{\max}(B)$. For a scalar function $f: \mathbb{R} \to \mathbb{R}$, the spectral reconstruction is $f(X) := U f(D) U^*$ via the eigendecomposition $X = UDU^*$, where $U^*$ denotes the adjoint of $U$. While scalar monotonicity does not generally imply operator monotonicity, the operator logarithm is monotone. In finite dimension, for any monotone scalar $f$, the trace composition $\tr \circ f$ is monotone. Furthermore, convexity is also preserved under trace composition, as in the following consequence of \citet[Theorem 2.7]{hansen2003jensen}. We state the trace inequalities for matrices and justify their use for operators by finite-rank approximation in Appendix~\ref{appendix:finite_rank_reduction}.
\begin{fact}[Operator-valued trace Jensen inequality] \label{fact:operator_trace_jensen} 
    If $f: \mathbb{R} \to \R$ is continuous and convex, then $X \mapsto \tr f(X)$ is convex on the self-adjoint matrices $\sma$.
\end{fact}




\paragraph{Lieb's concavity theorem.}
A pivotal tool in this framework is Lieb’s concavity theorem \citep{lieb1973convex, araki1975relative}, which provides a powerful extension of the Golden-Thompson inequality. While the scalar inequality $e^{a+b} \leq e^a e^b$ holds for all $a, b \in \mathbb{R}$, its matrix counterpart $\tr \exp(A+B) \leq \tr(\exp A \exp B)$ is significantly more difficult to manipulate because the operator $A+B$ does not generally commute with $A$ or $B$. Lieb's result addresses this structural obstacle by establishing the concavity of the map $A \mapsto \tr \exp(H + \log A)$, which effectively allows for the decoupling of expectations in the matrix Laplace transform. This result is the technical cornerstone for nearly all modern matrix concentration inequalities, as it enables the application of Markov-type tail bounds to the spectral radius of sums of independent operators.

\begin{fact} [Lieb's concavity theorem] \label{fact:liebs_concavity_theorem}
    For a fixed self-adjoint matrix $A \in \sma$, the map $X \mapsto \tr \exp ( \log X + A)$ is concave on the positive definite matrices.
\end{fact}

\paragraph{Intrinsic dimension.}
The focus of this contribution is to provide concentration inequalities that depend on the intrinsic dimension, as opposed to the ambient dimension. Given a nonzero positive trace-class operator $A$ on $\hs$, the intrinsic dimension is defined as
\begin{align*}
    r(A) := \frac{\tr(A)}{\|A\|}.
\end{align*}
Importantly, it holds that $1 \leq r(A) \leq \text{rank}(A) \leq d$, where $d=\dim(\hs)$ may be infinite. Results that replace the ambient dimension by the intrinsic dimension are strict improvements. Note that, in the finite-dimensional isotropic case $A = I$, it follows that $r(A) = d$; nonetheless, if all but one eigenvalue are zero, then $r(A) = 1$ even if $d$ is infinite.


\paragraph{Sub-$\psi$ processes.}
In order to present light-tailed operator concentration inequalities in a unified manner, we adopt a sub-$\psi$ notation consistent with the scalar literature \citep{howard2020time}. Throughout, a function $\psi$ should be understood as being like a cumulant-generating function (CGF). A real-valued function $\psi$ with domain $[0, \theta_{\max})$ is \emph{CGF-like} if it is strictly convex and twice continuously differentiable with $\psi(0) = \psi'_+(0) = 0$ and $\sup_{\theta \in [0, \theta_{\max})} \psi(\theta) = \infty$. As shown in the aforementioned work, this abstraction allows our results to easily be adapted to various noise distributions (including sub-Gaussian, sub-Poisson, and sub-exponential).

\paragraph{(Super)martingales.}
A filtration $\F = (\F_t)_{t \in \mathbb{N}}$ is a sequence of
$\sigma$-algebras such that $\F_t \subseteq \F_{t+1}$ for all $t$.
A stochastic process $M = (M_t)_{t \in \mathbb{N}}$ is said to be adapted to
$\F$ if $M_t$ is $\F_t$-measurable, and it is called \emph{predictable} if $M_t$ is $\F_{t-1}$-measurable. An integrable process $M$ is a supermartingale with respect to $\F$ if
$\E[M_{t+1} \mid \F_t] \leq M_t$, and a martingale if equality holds.
Inequalities between random variables are understood to hold almost surely,
and we use the shorthand $\E_t[\cdot] := \E[\cdot \mid \F_t]$.

\paragraph{The $\phi$ and $\varphi$ functions.}
The nonnegative functions
\begin{align}\label{eq:phi}
    \phi(u) := e^u - u - 1, \quad \varphi(u) := \cosh(u) - 1,
\end{align}
play a central role in this contribution. These functions are increasing on $[0, \infty)$ and decreasing on $(-\infty, 0]$. This implies that, for $r > 0$ and $\theta > 0$, the event $\{ \lambda_{\max} \lb  X \rb \geq r  \}$ is contained in  $\{ \lambda_{\max} \lb \phi (\theta X )\rb \geq \phi(\theta r)  \}$ (these are not necessarily equal given that $\lambda_{\max} \lb  X \rb$ can take negative values), and $\{ \| \varphi (\theta X )\| \geq \varphi(\theta r)  \} = \{ \|  X \| \geq r  \}$ (by symmetry of  $\varphi$). These event inclusions will be fundamental to the analyses. Importantly, both $\phi$ and $\varphi$ are also convex on $\R$, and so the following lemma applies to $\phi(S_n)$ and $\varphi(S_n)$, where $S_n$ is defined as in~\eqref{eq:sn_definition}. 
\begin{lemma} \label{lemma:submartingale}
    Let $X_1, \ldots, X_n$ be independent self-adjoint random matrices such that $\E X_i = 0$. For any continuous convex nonnegative $f$ and $\theta \in \R$, if $\E\tr f(\theta S_k)<\infty$ for every $k\leq n$, then $(\tr f(\theta S_k))_{k=0}^n$ is a nonnegative submartingale.
\end{lemma}
The proof follows immediately from Fact~\ref{fact:operator_trace_jensen} and conditional Jensen's inequality. For $f=\phi$ and $f=\varphi$, step~\ref{step:finite_rank_submartingale} in Appendix~\ref{appendix:finite_rank_reduction} extends this conclusion to compact operators whenever $\E\tr f(\theta S_n)<\infty$.
The preceding lemma enables the use of Doob’s maximal inequality, rather than the more elementary Markov inequality, to derive a partially time-uniform concentration bound. Moreover, suitably rescaled versions of  $\phi$ and $\varphi$ dominate the exponential function, which in turn will permit the derivation of closed-form results, in view of the easily verifiable bounds on $\phi + 1, \varphi + 1$:
\begin{align}\label{eq:bound_phi_e}
        \frac{e-1}{e} e^u \leq \phi(u) + 1, \quad \frac{e^u}{2} \leq \varphi(u) + 1 = \cosh(u).
\end{align}

    


\section{Master theorems}

We focus on the general problem of quantifying the concentration of the operator norm for sums of independent, light-tailed, centered random operators.  Providing bounds for centered operators allows for characterizing the concentration behavior of the operator noise (i.e., fluctuations around the mean). We target compact self-adjoint operators and work with centered functions whose traces are finite under our assumptions, as justified in Appendix~\ref{appendix:finite_rank_reduction}. We defer a discussion on martingale difference sequences to Section~\ref{section:further_results_martingale}.

We start by presenting the following master theorem, which serves as a unified framework for deriving specific concentration results, such as Hoeffding and Bennett-type inequalities. To maintain broad applicability, the theorem is formulated to be agnostic toward the specific growth characteristics of the random operators, captured here by the MGF-like function $\psi$.

\begin{theorem} \label{theorem:master_theorem}
Let $X_1, \ldots, X_n \in \csop$ be independent random operators such that $\E X_i = 0$, and recall $\phi$ and $\varphi$ from~\eqref{eq:phi}. If there exist $\sigma > 0$ and $d' \geq 1$ such that
    \begin{align} \label{eq:main_upper_bound}
        \E \tr(\phi (\theta S_n) ) + d'  \leq   d' \exp\lp \psi(|\theta|) \sigma^2 \rp
    \end{align}
    holds for any $\theta \in [0, \theta_{\max})$, then
    \begin{align} \label{eq:main_maximum_eigenvalue}
    \Pb \lp \sup_{i \leq n} \lambda_{\max}( S_i)  \geq r\rp \leq  \frac{e}{e-1} d' \lb\inf_{\theta \in [0, \theta_{\max})} \exp \lp \psi(\theta) \sigma^2 -\theta r \rp \rb.
    \end{align}
    If~\eqref{eq:main_upper_bound} holds for any $\theta \in (-\theta_{\max}, \theta_{\max})$, or if it holds with $\varphi$ in place of $\phi$, then
    \begin{align} \label{eq:main_operator_norm}
        \Pb \lp \sup_{i \leq n}\| S_i \| \geq r\rp \leq  2 d'\lb\inf_{\theta \in [0, \theta_{\max})} \exp \lp \psi(\theta) \sigma^2 -\theta r \rp \rb.
    \end{align}
\end{theorem}

\begin{proof}
    The terminal trace bound~\eqref{eq:main_upper_bound} gives the integrability required by step~\ref{step:finite_rank_submartingale} in Appendix~\ref{appendix:finite_rank_reduction}. Thus the finite-rank extension of Lemma~\ref{lemma:submartingale} justifies Doob's inequality for the centered traces used here. We prove~\eqref{eq:main_maximum_eigenvalue} and~\eqref{eq:main_operator_norm} separately. \paragraph{Proving~\eqref{eq:main_maximum_eigenvalue}.}
    For $\theta > 0$, it holds that
    \begin{align*}
        \Pb \lp \sup_{i \leq n} \lambda_{\max} (S_i) \geq r\rp &\stackrel{(i)}{=} \Pb \lp \sup_{i \leq n} \lambda_{\max} (\theta S_i) \geq \theta r\rp 
        \stackrel{(ii)}{\leq} \Pb \lp \sup_{i \leq n} \lambda_{\max} \lb \phi(\theta S_i) \rb \geq \phi(\theta r) \rp
        \\&\stackrel{(iii)}{\leq} \Pb \lp  \sup_{i \leq n} \tr(\phi (\theta S_i) ) \geq \phi(\theta r) \rp
        \\&\stackrel{(iv)}{\leq} \Pb \lp \sup_{i \leq n}\tr(\phi (\theta S_i) ) + d' \geq \phi(\theta r) + 1\rp
        \stackrel{(v)}{\leq}\frac{\E \tr(\phi (\theta S_n) ) + d' }{\phi(\theta r) + 1},
    \end{align*}
    where $(i)$ follows from $\theta > 0$, $(ii)$ follows from $\phi$ being strictly increasing on $\R_{>0}$,  $(iii)$ follows from the fact that $\phi$ is nonnegative, $(iv)$ follows from $d' \geq 1$, and $(v)$ follows from Lemma~\ref{lemma:submartingale} and Doob's inequality. In view of~\eqref{eq:main_upper_bound}, for $\theta \in [0, \theta_{\max})$, 
    \begin{align*}
        \E \tr(\phi (\theta S_n) ) + d'  \leq   d' \exp\lp \psi(\theta) \sigma^2 \rp ,
    \end{align*}
    as well as $\phi(\theta r) + 1  \geq \frac{e-1}{e} \exp(\theta r)$ by~\eqref{eq:bound_phi_e}, and so 
    \begin{align*}
         \Pb \lp \sup_{i \leq n} \lambda_{\max} (S_i) \geq r\rp 
         \leq d'\frac{\exp\lp \psi(\theta) \sigma^2 \rp }{\phi(\theta r) + 1} 
         \leq \frac{e}{e-1}d'\frac{\exp\lp \psi(\theta) \sigma^2 \rp }{\exp(\theta r)}.
    \end{align*}
    The inequality holds true for any arbitrary $\theta \in [0, \theta_{\max}) \cap \R_{>0}= (0, \theta_{\max})$. Inequality~\eqref{eq:main_maximum_eigenvalue} also holds trivially for $\theta = 0$.

\paragraph{Proving~\eqref{eq:main_operator_norm}.}
For $\theta > 0$, it holds that
    \begin{align*}
        \Pb \lp \sup_{i \leq n}\| S_i \| \geq r\rp &\stackrel{(i)}{=} \Pb \lp \sup_{i \leq n} \|\theta S_i\| \geq \theta r\rp 
        \stackrel{(ii)}{=} \Pb \lp\sup_{i \leq n} \varphi( \| \theta S_i \| ) \geq \varphi(\theta r) \rp
        \\&\stackrel{(iii)}{=} \Pb \lp \sup_{i \leq n} \| \varphi (\theta S_i) \| \geq \varphi(\theta r) \rp
        \stackrel{(iv)}{\leq} \Pb \lp  \sup_{i \leq n} \tr(\varphi (\theta S_i) ) \geq \varphi(\theta r) \rp
        \\&\stackrel{(v)}{\leq} \Pb \lp \sup_{i \leq n}\tr(\varphi (\theta S_i) ) + d' \geq \varphi(\theta r) + 1\rp
        \stackrel{(vi)}{\leq}\frac{\E \tr(\varphi (\theta S_n) ) + d'}{\varphi(\theta r) + 1},
    \end{align*}
    where $(i)$ follows from $\theta > 0$, $(ii)$ follows from $\varphi$ being strictly increasing on $\R_{>0}$, $(iii)$ follows from $\varphi$ being symmetric and increasing on $\R_{>0}$,  $(iv)$ follows from the fact that $\varphi$ is nonnegative, $(v)$ follows from $d'\geq1$, and $(vi)$ follows from Lemma~\ref{lemma:submartingale} and Doob's inequality. In view of~\eqref{eq:main_upper_bound}, for $\theta \in [0, \theta_{\max})$, 
    \begin{align*}
        \E \tr(\varphi (\theta S_n) ) + d' = \E \tr\lp \frac{\phi(\theta S_n) + \phi(-\theta S_n)}{2} \rp + d'  \leq   d' \exp\lp \psi(\theta) \sigma^2 \rp ,
    \end{align*}
    as well as $\varphi(\theta r) + 1  = \cosh(\theta r) \geq \frac{ \exp(\theta r)}{2}$, and so 
    \begin{align*}
         \Pb \lp \sup_{i \leq n} \| S_i \| \geq r\rp 
         \leq d'\frac{\exp\lp \psi(\theta) \sigma^2 \rp }{\varphi(\theta r) + 1} 
         \leq 2d'\frac{\exp\lp \psi(\theta) \sigma^2 \rp }{\exp(\theta r)}.
    \end{align*}
    The inequality holds true for any arbitrary $\theta \in [0, \theta_{\max}) \cap \R_{>0}= (0, \theta_{\max})$. Inequality~\eqref{eq:main_operator_norm} also holds trivially for $\theta = 0$.
\end{proof}

Theorem~\ref{theorem:master_theorem} will allow for improving and generalizing a variety of results. In particular, it strictly improves all the previous operator norm concentration inequalities we are aware of. For example, it sharpens the constants that would be obtained if generalizing the results from \citet{minsker2017some} and \citet{tropp2015introduction}, reducing them from $14$ and $8$ (respectively) to $2$. Furthermore,    Theorem~\ref{theorem:master_theorem} allows for recovering the best dimension-dependent inequalities \textit{up to constants}, but with the ambient dimension replaced by the intrinsic dimension. This master theorem closes the gap  between the dimension-dependent and dimension-free approaches; a thorough presentation is deferred to Section~\ref{section:bridging_gap}. As a byproduct, Theorem~\ref{theorem:master_theorem} also sharpens the existing ambient-dimension maximum eigenvalue inequalities, lowering the constants from $7$ \citep{minsker2017some} and $4$ \citep{tropp2015introduction} to $\frac{e}{e-1}\approx 1.5820$.

The proof of Theorem~\ref{theorem:master_theorem} is substantially inspired by that in \citet{minsker2017some} and \citet[Section 7]{tropp2015introduction}. In important contrast to these, we avoid using the trivial (and loose) upper bound $\exp\lp \psi(\theta) \sigma^2 \rp - 1 \leq \exp\lp \psi(\theta) \sigma^2 \rp$. Furthermore, we make a novel use of the hyperbolic cosine function, leading to a sharper concentration inequality than would follow from applying the union bound on the maximum and minimum eigenvalue. We also apply Doob's maximal inequality, as opposed to Markov's inequality, in order to obtain a partially time-uniform guarantee (that is, uniformly over $i$ for a fixed $n$).

Theorem~\ref{theorem:master_theorem} is very general. In order to prove a concentration inequality, one must show that~\eqref{eq:main_upper_bound} holds for a specific $\psi$. There is not a unique recipe for proving inequalities of the form~\eqref{eq:main_upper_bound}. For instance, we will later prove an operator-valued Hoeffding's inequality via the method of exchangeable pairs. However, a variety of concentration inequalities can be derived if the logarithm of the moment generating operators can be controlled, as it is exhibited in the following theorem. 

\begin{theorem} \label{theorem:iid_master_theorem}
    Let $X_1, \ldots, X_n \in \csop$ be a sequence of independent random operators such that $\E X_i = 0$. Assume that the $\Delta V_i\succeq0$ are deterministic trace-class operators, the MGFs exist as bounded operators, and
    \begin{align} \label{eq:mgf_assumption}
        \log \E \exp (\theta X_i ) \preceq \psi(|\theta|) \Delta V_i,
    \end{align}
    for all $\theta \in (-\theta_{\max}, \theta_{\max})$ and all $i \in \{ 1, \ldots n \}$. If
    \begin{align*}
        0<\sigma^2,\qquad \ldba V_n \rdba \leq \sigma^2 \leq \tr (V_n), \quad \text{where }V_n := \sum_{i \leq n}\Delta V_i,
    \end{align*}
    then~\eqref{eq:main_upper_bound} holds with $d'=\frac{\tr(V_n)}{\sigma^2}$. 
\end{theorem}

\begin{proof}
\allowdisplaybreaks
The finite-rank argument in Appendix~\ref{appendix:finite_rank_reduction} justifies applying Lieb's concavity theorem to the compressed matrices and passing to the full space. In particular,~\eqref{eq:finite_rank_trace_mgf_bound} gives, for $\theta\in(-\theta_{\max},\theta_{\max})$,
\begin{align*}
 \E\tr\phi(\theta S_n)\leq\tr\big(e^{\psi(|\theta|)V_n}-I\big)<\infty.
\end{align*}
     Now observe that 
    \begin{align*}
         \exp \lp \psi(|\theta|) V_n \rp - I &= \sum_{k \geq 1} \frac{\lp  \psi(|\theta|) V_n \rp^k}{k!} 
        = \psi(|\theta|) \lp V_n \rp^{\frac{1}{2}} \lb \sum_{k \geq 0} \frac{\lp  \psi(|\theta|)  V_n \rp^k}{(k+1)!} \rb \lp V_n \rp^{\frac{1}{2}} 
        \\&\preceq \psi(|\theta|) \ldba \sum_{k \geq 0} \frac{\lp  \psi(|\theta|) V_n \rp^k}{(k+1)!} \rdba V_n 
        \preceq \psi(|\theta|) \lb  \sum_{k \geq 0} \frac{\lp  \psi(|\theta|) \ldba V_n \rdba \rp^k}{(k+1)!}  \rb V_n 
        \\&\preceq \psi(|\theta|) \lb  \sum_{k \geq 0} \frac{\lp  \psi(|\theta|) \sigma^2 \rp^k}{(k+1)!}  \rb V_n 
        =  \lb  \sum_{k \geq 1} \frac{\lp  \psi(|\theta|) \sigma^2 \rp^k}{k!}  \rb \frac{V_n}{  \sigma^2}
        \\&=  \lb  \exp\lp \psi(|\theta|) \sigma^2 \rp - 1  \rb \frac{V_n}{  \sigma^2},
    \end{align*}
    Taking traces and adding $d'=\tr(V_n)/\sigma^2$ to both sides gives~\eqref{eq:main_upper_bound}, as required.
    
\end{proof}

The degenerate case $V_n=0$ in these variance-based bounds is immediate: $X_i=0$ almost surely, so all positive-threshold tail probabilities are zero.

\section{Concentration inequalities}

We present in this section specific instances of concentration inequalities, which are derived invoking Theorem~\ref{theorem:master_theorem} and Theorem~\ref{theorem:iid_master_theorem}, in conjunction with specific assumptions. We exhibit here operator-valued versions of Hoeffding, Bennett, and sub-Gaussian type inequalities, deferring Bernstein and sub-exponential type inequalities to Appendix~\ref{appendix:sub_exponential_inequality}. All proofs can be found in Appendix~\ref{appendix:proofs_corollaries}.

We begin by generalizing a Hoeffding-type inequality presented in \citet[Corollary 4.2]{mackey2014matrix},  from matrices to operators. We strictly improve the original operator norm inequality, replacing the ambient dimension by the intrinsic dimension while the constants remain unchanged.

\begin{corollary} [Operator-valued Hoeffding inequality] \label{theorem:hoeffding_iid_mackey}
    Let  $X_1, \ldots, X_n \in \csop$ be random and independent, and let $A_1,\ldots,A_n\in\csop$ be fixed, such that
    \begin{align*}
        \E X_i = 0, \quad X_i^2 \preceq A_i^2.
    \end{align*}
    If $0<\sigma^2$ and $\|V_n\| \leq \sigma^2 \leq \tr (V_n)<\infty$, where $V_n := \frac{1}{2}\sum_{i \leq n} \lp A_i^2 + \E X_i^2 \rp$, then it holds that
    \begin{align} \label{eq:hoeffding_max_eigenvalue}
    \Pb \lp \sup_{i \leq n} \lambda_{\max}( S_i)  \geq r\rp \leq  \frac{e}{e-1} \frac{\tr \lp V_n \rp}{  \sigma^2} \exp \lp - \frac{r^2}{2 \sigma^2} \rp ,
    \end{align}
    as well as
    \begin{align} \label{eq:hoeffding_operator_norm}
         \Pb \lp \sup_{i \leq n}\| S_i \| \geq r\rp \leq 2 \frac{\tr \lp V_n \rp}{ \sigma^2} \exp \lp - \frac{r^2}{2 \sigma^2} \rp.
    \end{align}
\end{corollary}

While the method of exchangeable pairs \citep{mackey2014matrix} yielded the former Hoeffding-type inequality, its strengths tend to manifest in somewhat different regimes than those addressed by the Lieb-based machinery developed in \citet{tropp2012user} and \citet{tropp2015introduction}. The two approaches thus offer complementary perspectives, each particularly well suited to certain structural assumptions and proof strategies. In what follows, we derive the remaining inequalities using Theorem~\ref{theorem:iid_master_theorem}, an approach aligned with the techniques found in \citet{minsker2017some} and \citet{tropp2015introduction}. We begin with sub-Gaussian noises, arguably the most ubiquitous tail assumption in the literature. 


\begin{corollary} [Operator-valued sub-Gaussian inequality] \label{theorem:sub_gaussian_iid}
    Assume the conditions of Theorem \ref{theorem:iid_master_theorem} hold with $\psi(\theta) = \psi_N(\theta):= \frac{\theta^2}{2}$ for $\theta \in \R$. Then, for all $r \geq 0$,
    \begin{align*}
        \Pb \lp \sup_{i \leq n} \lambda_{\max} ( S_i) \geq r\rp \leq \frac{e}{e-1} \frac{\tr \lp V_n \rp}{ \sigma^2} \exp \lp - \frac{r^2}{2\sigma^2} \rp,
    \end{align*}
    as well as
    \begin{align*}
        \Pb \lp \sup_{i \leq n}\| S_i \| \geq r\rp \leq 2 \frac{\tr \lp V_n \rp}{ \sigma^2} \exp \lp - \frac{r^2}{2\sigma^2} \rp.
    \end{align*}
\end{corollary}


 Of particular importance are also concentration inequalities for bounded random variables, as they are not only ubiquitous in practice, but they also serve as the foundation for unbounded cases (which are often analyzed by trimming or truncation). In the bounded setting, we provide the following Bennett-type inequality, which can be seen as a refined version of the so-called Bernstein's inequality \citep{boucheron2003concentration, vershynin2018high} for bounded random operators.

\begin{corollary} [Operator-valued Bennett's inequality] \label{theorem:bennett_iid}
    Assume the $X_i \in \csop$ are independent, with $\E X_i = 0$ and $\|X_i\| \leq c$ almost surely for all $i$. Define $\psi(\theta) = \psi_{P, c}(\theta):=\frac{e^{\theta c} - \theta c - 1}{c^2}$, $V_n := \sum_{i \leq n} \E X_i^2$, and $h(u) := (1+u)\log(1+u) - u$. If $0<\sigma^2$ and $\|V_n\|\leq\sigma^2\leq\tr(V_n)<\infty$, then
    \begin{align*} 
        \Pb \lp \sup_{i \leq n}\lambda_{\max}( S_i ) \geq r\rp \leq \frac{e}{e-1} \frac{\tr \lp V_n \rp}{ \sigma^2} \exp \lp - \frac{\sigma^2}{c^2} h\lp \frac{c r}{\sigma^2} \rp \rp,
    \end{align*}
    as well as
    \begin{align*} 
        \Pb \lp \sup_{i \leq n}\| S_i \| \geq r\rp \leq 2 \frac{\tr \lp V_n \rp}{ \sigma^2} \exp \lp - \frac{\sigma^2}{c^2} h\lp \frac{c r}{\sigma^2} \rp \rp.
    \end{align*}
\end{corollary}

We have stated the results for mean-zero random operators. Of course, they can be invoked for random operators centered on an arbitrary mean $\mu$, in order to yield confidence intervals for $\mu$. As it is recurrent with Bennett-type inequalities, Theorem~\ref{theorem:bennett_iid} cannot be \textit{tightly} inverted in \textit{closed-form} in order to yield a $(1-\delta)$-confidence interval, for $\delta \in (0, 1)$. However, by applying the well known deterministic inequality $ h(u) \geq \frac{u^2}{2\lp1 + \frac{u}{3} \rp}$
to Theorem~\ref{theorem:bennett_iid}, we obtain the $(1-\delta)$-confidence interval
\begin{align*} 
    \| \frac{1}{n}S_n - \mu \| \leq  \sigma\sqrt{\frac{2  }{n} \log \lp \frac{2}{\delta} \frac{\tr (V_n)}{\sigma^2} \rp} + \frac{c }{3n }\log \lp \frac{2}{\delta} \frac{\tr (V_n)}{\sigma^2} \rp.
\end{align*}

\section{Comparison to existing work: bridging the gap} \label{section:bridging_gap}


The ambient dimension-dependent approach from \citet{tropp2012user} and \citet{tropp2015introduction} relies on upper bounding $\lambda_{\max} [\exp(\theta S_n)]$ by $\tr [\exp(\theta S_n)]$. Note that, in order to upper bound the maximum eigenvalue by the trace, the function taken on $\theta S_n$ must be nonnegative, which the $\exp$ function clearly is. However, this upper bound is rather crude if $\theta S_n$ is small, given that for $\theta S_n \approx 0$, it follows that
\begin{align*}
    \lambda_{\max} [\exp(\theta S_n)] \approx \lambda_{\max} [I]= 1, \quad \tr [\exp(\theta S_n)] \approx \tr [I]= d.
\end{align*}
Because $\theta S_n$ is centered at $0$, the concentration inequalities pay an avoidable price on the dimension due to the use of the $\exp$ function.

The original intrinsic dimension-dependent approach \citep{minsker2017some} is based on taking a nonnegative function different to the $\exp$ function, such that it does not evaluate to $I$ at $0$. Seeking exponential decay of the tails, it is rather natural to consider other exponential-like functions as candidates. \citet{minsker2017some} opted for $\phi(u) = e^u - u - 1$. Nonetheless, inverting $\phi$ is rather complex, which complicates the analysis after applying Markov's (or Doob's) inequality. Perhaps more importantly, taking the $\exp$ function (as opposed to $\phi$) in the scalar Chernoff bound leads to optimal exponential rates (Cramér's theorem). 
However, the $\phi$ function behaves quite differently to the $\exp$ function, e.g., $\phi(0) = 0$ and $\exp(0) = 1$. 

Consequently, one of the cornerstones of our approach is to make our nonnegative function \textit{more similar} to the $\exp$ function. Intuitively, we can achieve this simply by adding one outside of the $\lambda_{\max}$ function, i.e.,
\begin{align*}
    \lambda_{\max} [\phi(\theta S_n)] + 1 \leq \tr [\phi(\theta S_n)] + 1 = 1 + \sum_{k = 2}^\infty \frac{\tr [(\theta S_n)^k]}{k!},
\end{align*}
so that the right-hand side resembles the $\exp$ function without the first order term (in the proofs we add $d'$, instead of $1$). While simple, this idea does not fall under the umbrella of the generalized matrix Laplace transform bound presented in \citet[Proposition 7.4.1]{tropp2015introduction}, resulting in the existing gap in the literature. Fundamentally, we used a more general matrix Laplace transform
\begin{align*}
    \Pb \lp  \lambda_{\max} (X) \geq r \rp \leq \frac{1}{\varkappa(r) + \kappa} \lb \kappa + \E \tr \varkappa(X) \rb, 
\end{align*}
where $\varkappa$ is a nonnegative function, and $\kappa \in \R_+$.  This idea can be further refined to derive operator norm bounds by taking $\varkappa = \varphi$ instead of $\varkappa = \phi$, leading to a strict improvement in the operator norm case. The $\cosh$ function has also been exploited in vector-valued concentration inequalities \citep{pinelis1992approach, pinelis1994optimum, martinez2024empirical, martinez2025vector}, but in a different technical fashion (the hyperbolic cosine
function was used in such contributions in conjunction with a integral-remainder Taylor-expansion
and the properties of a smooth Banach space). Nonetheless, both our work and such contributions exploit the upper bound $2 \cosh \geq \exp$, and the fact that the $\cosh$ Taylor expansion does not have a first order term.

\section{Concentration under martingale dependence} \label{section:further_results_martingale}

Operator-valued martingale difference sequences are substantially more delicate than sequences under independence,
as the intrinsic variance process becomes random and standard recursive arguments
based on Lieb's concavity theorem are no longer directly applicable.
We address this difficulty here in order to derive Freedman-type inequalities for the maximum eigenvalue of a martingale difference sequence. We derive the inequalities for the set of self-adjoint matrices $\sma$, so that Fact~\ref{fact:operator_trace_jensen} applies to functions taking negative values. In many case,  the inequalities can be extended to infinite dimensions, similarly to \citet[Section 3.2]{minsker2017some}, via Fatou's lemma.

We start by defining the truncation function $p(u) := \min(-u, 1)$,
which is concave as the minimum of two affine functions.
This function plays a central role in controlling the random intrinsic time.
Based on $p$, we introduce the auxiliary function
\begin{align}
    g(u) := e^u + p(u) - 1,
    \label{eq:g_definition}
\end{align}
which is increasing on $\R_{\geq0}$ and satisfies $0 \leq g \leq \phi$.
While the distinction between $g$ and $\phi$ is immaterial in the i.i.d.\ case,
the function $g$ leads to sharper control in the martingale setting,
as previously observed by \citet{minsker2017some}.
Analogously to~\eqref{eq:bound_phi_e}, we have
\begin{align} \label{eq:g_upper_bound}
    \frac{e-1}{e} e^u \leq g(u) + 1 \leq e^u + 1.
\end{align}

We now present our master inequality under martingale dependence using the sub-$\psi$ supermartingale formulation of~\cite{howard2020time}, leading to a Freedman-type inequality that holds up to a random intrinsic-time threshold $\lambda_{\max} (V_n) \leq \sigma^2$. 

\begin{theorem}
\label{theorem:master_theorem_martingale}
Let $X_1, X_2, \ldots \in \sma$ be a sequence of self-adjoint random matrices with $\E X_i = 0$ for all $i$,
such that
\begin{align}
    R_t = \tr \lb \exp \lp \theta S_t - \psi(\theta) V_t \rp \rb
    \label{eq:supermartingale_definition}
\end{align}
is a supermartingale for
$\theta \in [0, \theta_{\max})$.
Then, it holds that
\begin{align}
    \Pb \lp
    \lambda_{\max} (S_n) \geq r, \; \lambda_{\max} (V_n) \leq \sigma^2
    \rp
    \leq
    \frac{e}{e-1} \lc \tr\lb  p \lp -\psi(\theta)\E V_n  \rp \rb +1\rc \exp \lp \psi(\theta) \sigma^2 - \theta r \rp
    \label{eq:cosh_martingale},
\end{align}
for any $\theta \in [0, \theta_{\max})$.
\end{theorem}

\begin{proof}
Throughout, we shall assume that $\theta r - \psi(\theta) \sigma^2 \geq 0$, as~\eqref{eq:cosh_martingale} follows trivially otherwise. We start by noting that, if $\lambda_{\max} (S_n) \geq r$ and $\lambda_{\max} (V_n) \leq \sigma^2
$, then
\begin{align*}
    \lambda_{\max} \lp \theta S_n - \psi(\theta) V_n \rp \geq \lambda_{\max} \lp \theta S_n \rp - \lambda_{\max} \lp \psi(\theta) V_n \rp \geq \theta r - \psi(\theta) \sigma^2. 
\end{align*}
Thus, recalling the function $g$ from~\eqref{eq:g_definition}, $\Pb \lp
    \lambda_{\max} (S_n) \geq r, \; \lambda_{\max} (V_n) \leq \sigma^2
    \rp $ is upper bounded by 
\begin{align*}
    & \Pb \lp
    \lambda_{\max} \lp \theta S_n - \psi(\theta) V_n \rp \geq \theta r - \psi(\theta) \sigma^2
    \rp
    \stackrel{(i)}{\leq} \Pb \lp
    g\lp \lambda_{\max} \lp \theta S_n - \psi(\theta) V_n \rp \rp \geq g \lp \theta r - \psi(\theta) \sigma^2 \rp
    \rp
    \\= &\Pb \lp\lambda_{\max} 
    g\lp  \theta S_n - \psi(\theta) V_n  \rp \geq g \lp \theta r - \psi(\theta) \sigma^2 \rp
    \rp
    \stackrel{(ii)}{\leq} \Pb \lp\tr
    g\lp  \theta S_n - \psi(\theta) V_n \rp  \geq g \lp \theta r - \psi(\theta) \sigma^2 \rp
    \rp
    \\\stackrel{}{=} &\Pb \lp\tr 
    g\lp  \theta S_n - \psi(\theta) V_n \rp +1\geq g \lp \theta r - \psi(\theta) \sigma^2 \rp+1
    \rp
    \stackrel{(iii)}{\leq} \frac{\E \tr \lp
    g\lp  \theta S_n - \psi(\theta) V_n \rp \rp +1}{g \lp \theta r - \psi(\theta) \sigma^2 \rp+1}
    \\\stackrel{(iv)}{\leq} &\frac{e}{e-1}\lb \E \tr \lp
    g\lp  \theta S_n - \psi(\theta) V_n \rp \rp +1\rb \exp \lp \psi(\theta) \sigma^2 - \theta r \rp,
\end{align*}
where $(i)$ follows from $g$ being increasing on $\R_{\geq 0}$, $(ii)$ follows from $g$ being nonnegative, $(iii)$ follows from $g$ being nonnegative and Markov's inequality, and $(iv)$ follows from~\eqref{eq:g_upper_bound}. It remains to observe that, for $\theta \in [0,\theta_{\max})$,
\begin{align*}
    \E \tr \lb g\lp \theta S_n - \psi(\theta) V_n \rp \rb + 1&\leq \E \tr\lb p \lp \theta S_n - \psi(\theta) V_n \rp \rb+1 
    \leq  \tr \lb p \lp -\psi(\theta)\E V_n  \rp \rb+1, 
\end{align*}
where the first inequality follows from~\eqref{eq:supermartingale_definition} being a supermartingale, so $ \E \tr [ \exp( \theta S_n - \psi(\theta) V_n ) -I] \leq \tr [\exp(0)-I] = 0 $, and the second inequality from the concavity of $p$,  Fact~\ref{fact:operator_trace_jensen}, and $\E S_n = 0$.

\end{proof}

Theorem~\ref{theorem:master_theorem_martingale} is substantially inspired by the
martingale Bernstein inequality of \citet{minsker2017some}, which we compare to shortly. It extends to the operator norm (at a expense of a $2$ factor) by applying  on $\lambda_{\max} (S_n)$ and $\lambda_{\max} (-S_n)$, and union bounding the guarantees, provided the supermartingale assumption also holds for $-S_t$ with the same $\psi$ and $V_t$. Relative to the i.i.d.\ master theorem (Theorem~\ref{theorem:iid_master_theorem}),
the martingale bound incurs a mild inflation in the pre-exponential factor. 

The key object in Theorem~\ref{theorem:master_theorem_martingale} is the trace-exponential supermartingale~\eqref{eq:supermartingale_definition}. A broad class of such supermartingales was established in~\citet[Fact 1 and Lemma 3]{howard2020time}, including Hoeffding, Bernstein, Bennett, Bernoulli, and self-normalized type supermartingales. Thus, Theorem~\ref{theorem:master_theorem_martingale} is extensively applicable in conjunction with \citet{howard2020time}, and we recall numerous examples in Appendix~\ref{appendix:applicability_martingale_theorem} for completeness. As a concrete example, we derive an operator-valued Bernstein inequality, allowing direct comparison with \citet{minsker2017some}.

\begin{corollary} [Operator-valued Bernstein's inequality under martingale dependence] \label{theorem:martingale_sub_poisson}
     If Theorem~\ref{theorem:master_theorem_martingale} holds with $\psi(\theta) = \psi_{P, c}(\theta) := \frac{e^{c\theta} - c\theta - 1}{c^2}$ for $\theta \in [0, \infty)$, then, for any $r \geq 0$,
    \begin{align} \label{eq:our_martingale_bernstein_guarantee}
    \Pb \lp \lambda_{\max} (S_n) > r, \; \lambda_{\max}(V_n) \leq \sigma^2 \rp \leq \frac{e}{e-1} \lc 1 
    + \tr \lb p \lp - \frac{r}{c} \frac{\E V_n}{\sigma^2}\rp\rb \rc\exp \lp \frac{-r^2}{2\lp \sigma^2 + \frac{rc}{3} \rp} \rp.
\end{align}
\end{corollary}

In comparison, \citet[Theorem 3.2]{minsker2017some} established that
\begin{align} \label{eq:minsker_guarantee}
    \Pb \lp \lambda_{\max} (S_n) > r, \; \lambda_{\max}(V_n) \leq \sigma^2 \rp \leq 25 \tr \lb p \lp - \frac{r}{c} \frac{\E V_n}{\sigma^2}\rp\rb \exp \lp \frac{-r^2}{2\lp \sigma^2 + \frac{rc}{3} \rp} \rp,
\end{align}
for $r \geq \frac{1}{6} (c + \sqrt{c^2 + \sigma^2})$. In contrast, Corollary~\ref{theorem:martingale_sub_poisson}  holds for every $r$, and leads to sharper inequalities if $\sigma^2 $ is not too far from $\lambda_{\max} (\mathbb{E} V_n)$.


\section{Conclusion}

We have developed a unified framework for deriving concentration inequalities
for sums of self-adjoint random operators that depend on an intrinsic notion of
dimension.
Our approach yields a single master inequality from which a wide range of
classical and modern concentration results follow as special cases, including
sub-Gaussian, Hoeffding, Bernstein, Bennett, and sub-exponential inequalities.
In all cases, the resulting bounds strictly improve existing intrinsic and
ambient dimension inequalities, either by sharpening constants or reducing
logarithmic factors.

A central contribution of this work is the realization that the generalized matrix Laplace transform bound presented in \citet[Proposition 7.4.1]{tropp2015introduction} has to be further generalized in order to develop sharper inequalities in the intrinsic dimension scenario. Furthermore, we coupled such an idea with the use of the hyperbolic cosine function, deviating from using $\phi$ in order to obtain operator-norm concentration guarantees. The combination of the two offers a perspective that clarifies the relationship between several previously disparate results in the literature, closing the gap between dimension-independent and dimension-dependent regimes. From a technical standpoint, our proofs are comparatively short and modular, relying on a small collection of operator-analytic tools and avoiding case-by-case arguments.

Several directions for future work remain. We  extended our framework to martingale difference sequences of self-adjoint matrices in Section~\ref{section:further_results_martingale}.
By constructing nonnegative trace-exponential supermartingales, we obtained Freedman-type maximal inequalities under martingale dependence.
These results incur a mild inflation of the logarithmic
factor compared to the independence setting, similarly to \citet{minsker2017some}. Whether this inflation is an artifact of the proof technique or an inherent feature of intrinsic-dimension martingale inequalities remains an open question, and resolving it constitutes a natural direction for future work.

A second important direction concerns the optimality of exponential tail bounds derived via the Chernoff method.
In the scalar setting, it is well known that Chernoff-type arguments can be suboptimal by a multiplicative factor, and refined inequalities \citep{talagrand1995missing, bentkus2002remark} demonstrate that this loss can be avoided through more delicate concentration techniques. At present, analogous refinements are largely missing in the matrix-valued setting, where most known results rely on trace-exponential arguments. It would be of considerable interest to investigate how the so-called ``missing factor'' can be effectively recovered in matrix-valued settings.

\section*{Acknowledgements}
DMT thanks Lester Mackey and Stanislav Minsker for insightful feedback. AR was funded by NSF grant DMS-2310718.

\bibliographystyle{apalike}
\bibliography{bib}

\appendix

\section{Further operator-valued inequalities}\label{appendix:sub_exponential_inequality}

While sub-Gaussianity provides a powerful framework, its tail-decay requirements are fairly restrictive. It is thus natural to consider relaxations of it. Consequently, we examine random operators that attain the Bernstein condition
\begin{align} \label{eq:bernstein_condition}
        \E \lb  \lba X_i\rba^k \rb  \preceq \frac{1}{2} k! c^{k-2} \Delta V_i, 
\end{align}
for all integers $k\geq2$ and $c > 0$, with deterministic positive trace-class $\Delta V_i$. This offers a more flexible tail assumption while maintaining tractable concentration properties. A standard sufficient requirement for Bernstein’s condition is that the summands be bounded; for instance, the constraint $\| X_i \| \leq c$ directly implies \eqref{eq:bernstein_condition} for $\Delta V_i = \E X_i^2$, provided these variance operators are trace-class. Importantly, the Bernstein condition further accommodates a variety of unbounded distributions, a property that significantly extends its applicability. The following corollary exhibits the Bernstein-type inequalities.


\begin{corollary} [Operator-valued Bernstein's inequality] \label{theorem:bernstein_iid}
   Assume the $X_i \in \csop$ are independent, with $\E X_i = 0$ and~\eqref{eq:bernstein_condition} for all $i$, and define $V_n = \sum_{i \leq n} \Delta V_i$. If $0<\sigma^2$ and $\sigma^2 \in [\| V_n\|, \tr(V_n)]$, then
   \begin{align*}
        \Pb \lp \sup_{i \leq n} \lambda_{\max}( S_i ) \geq r\rp \leq \frac{e}{e-1} \frac{\tr \lp V_n \rp}{ \sigma^2} \exp \lp \frac{-r^2}{2(\sigma^2 + cr)} \rp,
    \end{align*}
    as well as
    \begin{align*}
        \Pb \lp \sup_{i \leq n}\| S_i \| \geq r\rp \leq 2 \frac{\tr \lp V_n \rp}{ \sigma^2} \exp \lp \frac{-r^2}{2(\sigma^2 + cr)} \rp.
    \end{align*}
\end{corollary}

\begin{proof}
For $|\theta|<1/c$,
\begin{align*}
    \E \exp (\theta X_i ) &= I + \E \lb \sum_{k = 2}^\infty \frac{(\theta X_i )^k}{k!}\rb \preceq I +  \sum_{k = 2}^\infty |\theta|^k \frac{ \E \lb| X_i |^k \rb}{k!} 
    \preceq I +  \frac{\theta^2}{2} \lb \sum_{k = 0}^\infty (|\theta|c)^{k} \rb  \Delta V_i.
\end{align*}
Furthermore, $\sum_{k = 0}^\infty (|\theta|c)^{k} = \frac{1}{1 - c |\theta|}$ for $|\theta|<1/c$. Thus, by monotonicity of the operator $\log$ function,
\begin{align*}
    \log \E \exp (\theta X_i ) &\preceq \log \lp I + \frac{\theta^2}{2(1 - c |\theta|)} \Delta V_i \rp \preceq \log \lp \exp \lp \frac{\theta^2}{2(1 - c |\theta|)} \Delta V_i \rp \rp
    = \frac{\theta^2}{2(1 - c |\theta|)} \Delta V_i.
\end{align*}
 It now suffices to apply Theorem \ref{theorem:iid_master_theorem} with $\psi(\theta)= \psi_{G, c}(\theta) = \frac{\theta^2}{2(1 - c\theta)}$ and $\theta = \frac{r}{\sigma^2 + cr}$.
\end{proof}


For scalar random variables, it is well-known that the Bernstein condition implies sub-exponential tails \citep{wainwright2019high}. In fact, we proved Theorem~\ref{theorem:bernstein_iid} via establishing a sub-Gamma tail condition for the random operators (i.e., $\psi = \psi_{G, c}$), and \citet[Appendix E]{howard2021time} proved sub-Gamma and sub-exponential tails to be equivalent. Nevertheless, the direct sub-exponential parameterization sometimes provides a more convenient framework for tail analysis than the Bernstein condition. For the sake of theoretical completeness and practical utility, we provide an explicit operator-valued sub-exponential inequality.

\begin{corollary} [Operator-valued sub-exponential inequality] \label{theorem:sub_exponential_iid}
    If there exist constants $\nu, \alpha > 0$ such that Theorem~\ref{theorem:iid_master_theorem} holds with $\psi(\theta) = \psi_E(\theta) := \frac{\theta^2 \nu^2}{2 \sigma^2}$ for all $\theta \in [0, 1/\alpha)$, then
    \begin{align*}
        \Pb \lp \sup_{i \leq n}\lambda_{\max} ( S_i) \geq r\rp \leq \frac{e}{e-1} \frac{\tr \lp V_n\rp}{ \sigma^2} \exp \lp - \frac{1}{2} \min \lp \frac{r^2}{\nu^2}, \frac{r}{\alpha} \rp \rp,
    \end{align*}
    as well as
    \begin{align*}
        \Pb \lp \sup_{i \leq n}\| S_i \| \geq r\rp \leq 2 \frac{\tr \lp V_n\rp}{ \sigma^2} \exp \lp - \frac{1}{2} \min \lp \frac{r^2}{\nu^2}, \frac{r}{\alpha} \rp \rp.
    \end{align*}
\end{corollary}

\begin{proof}
It suffices to apply Theorem \ref{theorem:iid_master_theorem} with $\psi(\theta) = \psi_E(\theta) := \frac{\theta^2 \nu^2}{2 \sigma^2}$. Take $\theta = r/\nu^2$ if $r/\nu^2 < 1/\alpha$, and $\theta \to 1/\alpha$ otherwise.
\end{proof}

\section{Proofs of corollaries} \label{appendix:proofs_corollaries}

\subsection{Proof of Corollary~\ref{theorem:hoeffding_iid_mackey}}

If we can prove the bound
\begin{align} \label{eq:upper_bound_l}
 l(\theta):=\E\tr\phi(\theta S_n)+\frac{\tr(V_n)}{\sigma^2}
 \leq\frac{\tr(V_n)}{\sigma^2}\exp\lp\frac{\theta^2\sigma^2}{2}\rp,
\end{align}
then an invocation of Theorem~\ref{theorem:master_theorem} with $d'=\tr(V_n)/\sigma^2$ and $\psi(\theta)=\theta^2/2$ will conclude the proof. We prove~\eqref{eq:upper_bound_l} similarly to \citet{mackey2014matrix}, first on finite-dimensional compressions so that all the intermediate traces are finite.

Use the projections $P_m$ and sums $S_{n,m}$ from Appendix~\ref{appendix:finite_rank_reduction}. Work first with $\theta\geq0$, and define
\begin{align*}
 l_m(\theta)&:=\E\tr\phi(\theta S_{n,m})+\frac{\tr(V_n)}{\sigma^2},
 \qquad V_{n,m}:=P_mV_nP_m,\\
 \Delta_{n,m}&:=\frac12\sum_{i\leq n}\big(X_{i,m}^2+\E X_{i,m}^2\big)
 \preceq V_{n,m}.
\end{align*}
The last inequality follows from $X_{i,m}^2\preceq P_mX_i^2P_m\preceq P_mA_i^2P_m$. Also, $\|V_{n,m}\|\leq\sigma^2$ and $\tr(V_{n,m})\leq\tr(V_n)$. Applying \citet[Lemma 3.7 and Equation (2.6)]{mackey2014matrix} to the compressed matrices gives
\begin{align*}
 \E\tr\big[S_{n,m}e^{\theta S_{n,m}}\big]
 \leq\theta\E\tr\big[\Delta_{n,m}e^{\theta S_{n,m}}\big].
\end{align*}
Taking the derivative of $l_m(\theta)$, we can move the derivative inside the expectation by dominated convergence and the deterministic bound $\|S_{n,m}\|\leq\sum_{i\leq n}\|A_i\|$. Since $\E S_{n,m}=0$, we obtain
\begin{align*}
 l_m'(\theta)
 &=\E\tr\big[S_{n,m}e^{\theta S_{n,m}}-S_{n,m}\big]
 =\E\tr\big[S_{n,m}e^{\theta S_{n,m}}\big]\\
 &\leq\theta\E\tr\big[\Delta_{n,m}e^{\theta S_{n,m}}\big]
 \leq\theta\E\tr\big[V_{n,m}e^{\theta S_{n,m}}\big]\\
 &=\theta\E\tr\big[V_{n,m}\phi(\theta S_{n,m})\big]
   +\theta\tr(V_{n,m})\\
 &\leq\theta\|V_{n,m}\|\E\tr\phi(\theta S_{n,m})
   +\theta\tr(V_{n,m})
 \leq\theta\sigma^2 l_m(\theta).
\end{align*}
Here the equality in the third line uses $\E S_{n,m}=0$, and the next inequality uses $\phi(\theta S_{n,m})\succeq0$. In view of the fundamental theorem of calculus,
\begin{align*}
 \log l_m(\theta)
 &=\log l_m(0)+\int_0^\theta\frac{l_m'(s)}{l_m(s)}\,ds\\
 &\leq\log\lp\frac{\tr(V_n)}{\sigma^2}\rp
       +\int_0^\theta s\sigma^2\,ds
 =\log\lp\frac{\tr(V_n)}{\sigma^2}\rp+\frac{\theta^2\sigma^2}{2}.
\end{align*}
Thus~\eqref{eq:upper_bound_l} holds with $l_m$ in place of $l$. The assumption $\tr(V_n)<\infty$ implies $\tr(A_i^2)<\infty$, and $X_i^2\preceq A_i^2$ makes each summand Hilbert-Schmidt. The almost sure trace convergence in~\eqref{eq:finite_rank_trace_convergence} and Fatou's lemma now give~\eqref{eq:upper_bound_l} on the full space. Applying the same argument to $-X_i$ proves the bound for $\theta<0$ as well.

\subsection{Proof of Corollary~\ref{theorem:sub_gaussian_iid}}

It suffices to apply Theorem \ref{theorem:iid_master_theorem} with $\psi(\theta)= \psi_N(\theta) = \frac{\theta^2}{2}$ and $\theta = \frac{r}{\sigma^2}$.

\subsection{Proof of Corollary~\ref{theorem:bennett_iid}}

 For $\theta \in (-\theta_{\max}, \theta_{\max})$, 
\begin{align*}
    \E \exp (\theta X_i ) &= \E \lb \sum_{k \geq 0} \frac{(\theta X_i )^k}{k!} \rb = I + \E \lb \sum_{k \geq 2} \frac{(\theta X_i )^k}{k!} \rb 
    = I +  \sum_{k \geq 2} \E \lb (\theta X_i )^2\frac{(\theta X_i )^{k-2}}{k!} \rb  
    \\&\preceq I +  \sum_{k \geq 2} \E \lb (\theta X_i )^2\frac{(|\theta| \|X_i\| )^{k-2}}{k!} \rb \preceq I +  \sum_{k \geq 2} \E \lb (\theta X_i )^2\frac{(|\theta| c )^{k-2}}{k!} \rb
    \\&=  I + \frac{\E \lp  X_i ^2 \rp}{c^2} \sum_{k \geq 2}  \frac{(|\theta| c )^{k}}{k!} =  I + \E \lp  X_i ^2 \rp \psi_{P, c}(|\theta|) \preceq \exp \lp \E \lp  X_i ^2 \rp \psi_{P, c}(|\theta|) \rp,
\end{align*}
where the last inequality follows from $1+x \leq \exp(x)$ for all $x\in \R$. Given the monotonicity of the operator $\log$ function, we obtain that $ \log \E \exp (\theta X_i ) \preceq \E \lp  X_i ^2 \rp \psi_{P, c}(|\theta|)$.  It now suffices to apply Theorem~\ref{theorem:iid_master_theorem} with $\psi(\theta)= \psi_{P, c}(\theta)$ and $\theta = \frac{1}{c} \log \lp 1 + \frac{c r}{\sigma^2} \rp$.

\subsection{Proof of Corollary~\ref{theorem:martingale_sub_poisson}}

It suffices to invoke Theorem~\ref{theorem:master_theorem_martingale} with $\theta^* = \frac{1}{c} \log \lp 1 + \frac{c r}{\sigma^2} \rp$, and observe that
\begin{align} \label{eq:loose_bound_in_martingale_bernstein}
    \tr \lb  p \lp - \psi_{P, c}\lp \theta^*  \rp \E V_n \rp \rb \leq \tr \lb p \lp -\frac{r}{c} \frac{\E V_n}{\sigma^2} \rp \rb,
\end{align}
which was proved in \citet{minsker2017some}.

\section{Applicability of master theorem under martingale dependence} \label{appendix:applicability_martingale_theorem}

We elucidate here the wide applicability of Theorem~\ref{theorem:master_theorem_martingale}. Following the notation from  \citet{howard2020time}, we denote 
\begin{align*}
    [S]_t &:= \sum_{i = 1}^t X_i^2, \quad \la S\ra_t := \sum_{i = 1}^t \E_{i-1} X_i^2,\\
    [S_+]_t &:= \sum_{i = 1}^t \max (0, X_i )^2, \quad \la S_- \ra_t := \sum_{i = 1}^t \E_{i-1} \min(0, X_i)^2.
\end{align*}
We further remind the reader the definition of the specific $\psi$-functions
\begin{align*}
    \psi_{N}(\theta) = \frac{\theta^2}{2}, \quad \psi_{P, c}(\theta) = \frac{e^{c\theta} - c\theta - 1}{c^2}, \quad  \psi_{G, c}(\theta) = \frac{\theta^2}{2(1 - c\theta)}.
\end{align*}
A variety of supermartingale constructions of the form~\eqref{eq:supermartingale_definition} can be proven in view of the following lemma, which is a simplification of \citet[Lemma 4]{howard2020time}.
\begin{lemma} \label{lemma:howard_lemma4_simplified}
    If
    \begin{align} \label{eq:condition_lemma_4_howard}
        \log \E_{t-1} \exp \lp  \lambda X_t - \psi(\lambda) \Delta U_t\rp \preceq \psi(\lambda) \Delta W_t
    \end{align}
    then~\eqref{eq:supermartingale_definition} is a supermartingale with $V_t := U_t + W_t$.
\end{lemma}

To begin with, replacing $\E$ by $\E_{t-1}$ in the proofs of Corollary~\ref{theorem:bennett_iid} and Corollary~\ref{theorem:bernstein_iid}  (that is, going from independence to martingale dependence) precisely yields a guarantee of the form~\eqref{eq:condition_lemma_4_howard} in both cases. Thus, we immediately obtain the following results in view of Lemma~\ref{lemma:howard_lemma4_simplified}.
\begin{itemize}
    \item (Bernstein's inequality) If $\E_{t-1} X_t = 0$ and $\E_{t-1} \lb  \lba X_t\rba^k \rb  \preceq \frac{1}{2} k! c^{k-2} \Delta V_t$, then~\eqref{eq:supermartingale_definition} is a supermartingale with $\psi = \psi_{G, c}$ and $V_t = \sum_{i \leq t} \Delta V_i$.
    \item (Bennett's inequality) If $\E_{t-1} X_t = 0$ and $\|X_i\| \leq c$, then~\eqref{eq:supermartingale_definition} is a supermartingale with $\psi = \psi_{G, c}$ and $V_t = \la S \ra_t$.\footnote{Our Corollary~\ref{theorem:martingale_sub_poisson} is actually an invocation of Bennett's inequality in conjunction with the upper bound~\eqref{eq:loose_bound_in_martingale_bernstein}.}
\end{itemize}
Furthermore, \citet[Lemma 3]{howard2020time} proved that Lemma~\ref{lemma:howard_lemma4_simplified} can be applied to a variety of cases, implying then that~\eqref{eq:supermartingale_definition} is a supermartingale in the following scenarios.
\begin{itemize}
    \item (Conditionally symmetric, \citet[Lemma 3(d)]{howard2020time}) If $\E_{t-1} X_t = 0$ and $X_t \sim - X_t | \F_{t-1}$, then~\eqref{eq:supermartingale_definition} is a supermartingale with $\psi = \psi_N$ and $V_t = [S]_t$.
    \item (General self-normalized I, \citet[Lemma 3(f)]{howard2020time}) If $\E_{t-1} X_t = 0$ and $\E_{t-1} X_t^2$ is finite, then~\eqref{eq:supermartingale_definition} is a supermartingale with $\psi = \psi_N$ and $V_t = \frac{([S]_t + 2 \la S \ra_t)}{3}$.
    \item (General self-normalized II, \citet[Lemma 3(g)]{howard2020time}) If $\E_{t-1} X_t = 0$ and $\E_{t-1} X_t^2$ is finite, then~\eqref{eq:supermartingale_definition} is a supermartingale with $\psi = \psi_N$, $V_t = [S_+]_t + \la S_- \ra_t$.
    \item (Hoeffding, \citet[Lemma 3(h)]{howard2020time}) If $\E_{t-1} X_t = 0$ and $X_t^2 \preceq A_t^2$, then~\eqref{eq:supermartingale_definition} is a supermartingale with $\psi = \psi_N$ and $V_t = \sum_{i \leq t} A_i^2$.
    \item (Cubic self-normalized, \citet[Lemma 3(i)]{howard2020time}) If $\E_{t-1} X_t = 0$ and $ \E_{t-1} |X_t|^3$ is finite, then~\eqref{eq:supermartingale_definition} is a supermartingale with $\psi = \psi_{G, c}$, $c = \frac{1}{6}$, and $V_t = [S]_t + \sum_{i \leq t} \E_{i-1} |X_i|^3$.
\end{itemize}

\section{Finite-rank reduction and trace domains} \label{appendix:finite_rank_reduction}

The proofs of Theorems~\ref{theorem:master_theorem} and~\ref{theorem:iid_master_theorem} use trace Jensen and Lieb's concavity, but Facts~\ref{fact:operator_trace_jensen} and~\ref{fact:liebs_concavity_theorem} are stated for matrices. In an infinite-dimensional space, even the identity has infinite trace, so we must first check which expressions can be traced before passing from matrices to operators. The centered functions $\phi(A)=e^A-A-I$, $\varphi(A)=\cosh(A)-I$, and $e^{\psi V}-I$ allow us to make this passage.

For the independent mean-zero compact operators considered in the master theorems, we establish the following points in order:
\begin{samepage}
\begin{enumerate}
    \item \label{step:finite_rank_traces} The functions $\phi$ and $\varphi$ have finite traces on Hilbert-Schmidt operators, and finite-rank approximation recovers these centered traces.
    \item \label{step:finite_rank_mgf} The MGF assumption~\eqref{eq:mgf_assumption} ensures that the summands are Hilbert-Schmidt and provides the integrability needed for the compressed MGFs and their logarithms to converge in trace norm.
    \item \label{step:finite_rank_lieb} Lieb's inequality passes to the full space, bounding $\E\tr\phi(\pm\theta S_n)$ and $\E\tr\varphi(\theta S_n)$ by the finite quantity $\tr(e^{\psi(|\theta|)V_n}-I)$.
    \item \label{step:finite_rank_submartingale} If $\E\tr f(\theta S_n)<\infty$ for $f=\phi$ or $f=\varphi$, then $(\tr f(\theta S_k))_{k=0}^n$ is an integrable nonnegative submartingale. This step also applies under~\eqref{eq:main_upper_bound}, without the MGF assumption.
\end{enumerate}
\end{samepage}

\textit{Details of step~\ref{step:finite_rank_traces}.}
We begin by explaining why these centered traces are finite. Write $\|T\|_1=\tr|T|$, where $|T|=(T^*T)^{1/2}$, for the trace norm. An operator is trace-class when this norm is finite. For self-adjoint Hilbert-Schmidt $A$, its square is trace-class, and
\begin{align} \label{eq:finite_rank_series}
 \phi(A)=\sum_{q\geq2}\frac{A^q}{q!},
 \qquad \|A^q\|_1\leq\|A\|^{q-2}\|A\|_{\mathrm{HS}}^2,\quad q\geq2.
\end{align}
Here we use $A^q=A^{q-2}A^2$ and $\|BC\|_1\leq\|B\|\|C\|_1$, with $\|A^2\|_1=\tr(A^2)=\|A\|_{\mathrm{HS}}^2$. Consequently, this series converges in trace norm, as does $\varphi(A)=(\phi(A)+\phi(-A))/2$. The absence of constant and linear terms is what allows the Hilbert-Schmidt assumption to suffice. Similarly, if $V\succeq0$ is trace-class and $\psi\geq0$ is finite, then $\|V^j\|_1\leq\|V\|^{j-1}\tr(V)$ shows that $e^{\psi V}-I$ is trace-class. All infinite-dimensional traces of exponential expressions below are taken after this centering. Trace-norm convergence is useful here because $|\tr C-\tr D|\leq\|C-D\|_1$ for trace-class $C,D$.

To reduce the infinite-dimensional case to matrices, choose an orthonormal basis of $\hs$ and let $P_m$ project onto its first $m$ vectors, following \citet{minsker2017some}. These projections increase to the identity: $P_mx\to x$ for every $x\in\hs$. Define
\begin{align*}
 X_{i,m}=P_mX_iP_m,\qquad
 S_{k,m}=\sum_{i\leq k}X_{i,m}=P_mS_kP_m.
\end{align*}
On $P_m\hs$, the $X_{i,m}$ are finite-dimensional self-adjoint matrices and retain independence and zero mean. We also view them as operators on $\hs$ by setting them equal to zero on $(P_m\hs)^\perp$. Since $\phi(0)=\varphi(0)=0$, their centered functions are zero there as well, so their matrix traces agree with their traces on the full space.

For a self-adjoint Hilbert-Schmidt operator $A$, the compression $A_m=P_mAP_m$ retains the first $m$ rows and columns of its matrix in the chosen basis. Square summability of the matrix entries gives $\|A_m-A\|_{\mathrm{HS}}\to0$, with $\|A_m\|_{\mathrm{HS}}\leq\|A\|_{\mathrm{HS}}$. To see how this yields convergence of centered traces, use the product inequality $\|BC\|_1\leq\|B\|_{\mathrm{HS}}\|C\|_{\mathrm{HS}}$ to obtain
\begin{align*}
 \|A_m^2-A^2\|_1
 &\leq\big(\|A_m\|_{\mathrm{HS}}+\|A\|_{\mathrm{HS}}\big)
       \|A_m-A\|_{\mathrm{HS}}\longrightarrow0.
\end{align*}
For higher powers, the identity
\begin{align*}
 A_m^q-A^q=\sum_{\ell=0}^{q-1}A_m^\ell(A_m-A)A^{q-1-\ell},\qquad q\geq2,
\end{align*}
has $A_m-A$ and at least one further Hilbert-Schmidt factor in each term. The same product inequality therefore gives trace-norm convergence of every power of degree at least two. The series bound in~\eqref{eq:finite_rank_series} is uniform in $m$, since $\|A_m\|\leq\|A\|$. We can consequently pass the limit through the power series for $\phi$ and $\varphi$. Whenever the summands are Hilbert-Schmidt, for every fixed $\theta$, $k\leq n$, and $f\in\{\phi,\varphi\}$,
\begin{align} \label{eq:finite_rank_trace_convergence}
 \|f(\theta S_{k,m})-f(\theta S_k)\|_1\longrightarrow0
 \quad\text{almost surely}.
\end{align}

\textit{Details of step~\ref{step:finite_rank_mgf}.}
Passing to expectations requires an integrable bound in addition to the samplewise convergence in~\eqref{eq:finite_rank_trace_convergence}. We now show that the MGF assumption supplies this bound and ensures that the summands are Hilbert-Schmidt. The case $\theta=0$ is immediate, so fix $0<|\theta|<\theta_{\max}$ and put $M_i^\pm=\E e^{\pm\theta X_i}$. The scalar inequality $e^u\geq1+u$, applied by spectral calculus and then averaged, gives $M_i^\pm\succeq I$ because $\E X_i=0$. Thus~\eqref{eq:mgf_assumption} implies
\begin{align*}
 0\preceq\log M_i^\pm\preceq\psi(|\theta|)\Delta V_i.
\end{align*}
The upper bound is trace-class, so $\log M_i^\pm$ is also trace-class. The centered-exponential argument following~\eqref{eq:finite_rank_series} now gives that $M_i^\pm-I=e^{\log M_i^\pm}-I$ is trace-class. Since $\varphi(\theta X_i)\succeq0$, Tonelli's theorem, applied to its diagonal entries in an orthonormal basis, yields
\begin{align} \label{eq:finite_rank_increment_integrability}
 \E\tr\varphi(\theta X_i)
 =\frac12\tr\big[(M_i^+-I)+(M_i^--I)\big]<\infty.
\end{align}
Moreover, $\varphi(\theta X_i)\succeq\theta^2X_i^2/2$, so~\eqref{eq:finite_rank_increment_integrability} gives $\E\|X_i\|_{\mathrm{HS}}^2<\infty$. In particular, the summands are almost surely Hilbert-Schmidt, and~\eqref{eq:finite_rank_trace_convergence} applies.

We next use the integrability in~\eqref{eq:finite_rank_increment_integrability} to control the compressed MGFs. Define the nonnegative convex functions $f_\pm(u)=\phi(\pm\theta u)=e^{\pm\theta u}-1\mp\theta u$, which satisfy $f_+(u)+f_-(u)=2\varphi(\theta u)$. The following consequence of scalar Jensen explains why compression does not increase their traces. For any nonnegative convex $f$ with $f(0)=0$ and compact self-adjoint $A$, choose an orthonormal eigenbasis $(u_j)_{j=1}^m$ of $A_m$ on $P_m\hs$. Its eigenvalues are $\langle u_j,Au_j\rangle$. Expressing $u_j$ in an eigenbasis of $A$ writes this inner product as a weighted average of the eigenvalues of $A$, with nonnegative weights summing to one. Applying scalar Jensen to each of these averages and then summing gives
\begin{align} \label{eq:finite_rank_trace_compression}
 \tr f(A_m)
 &=\sum_{j=1}^m f(\langle u_j,Au_j\rangle)
 \leq\sum_{j=1}^m\langle u_j,f(A)u_j\rangle
 \leq\tr f(A).
\end{align}
The last inequality uses $f(A)\succeq0$, and remains valid if its trace is infinite. Applied to $f_\pm$ and $A=X_i$, this gives
\begin{align*}
 0\leq\tr f_\pm(X_{i,m})\leq\tr f_\pm(X_i)
 \leq2\tr\varphi(\theta X_i).
\end{align*}
Together with the trace-norm convergence in~\eqref{eq:finite_rank_trace_convergence}, applied to $\phi$ with both signs of $\theta$, this supplies the integrable bound
$\|f_\pm(X_{i,m})-f_\pm(X_i)\|_1\leq4\tr\varphi(\theta X_i)$.
Dominated convergence therefore allows us to pass expectations through the compression limit.

More explicitly, let $M_{i,m}^\pm=\E e^{\pm\theta X_{i,m}}$, extended by the identity on $(P_m\hs)^\perp$. Zero mean gives $M_{i,m}^\pm-I=\E f_\pm(X_{i,m})$ and $M_i^\pm-I=\E f_\pm(X_i)$, hence, by dominated convergence,
\begin{align*}
 \|M_{i,m}^\pm-M_i^\pm\|_1
 \leq\E\|f_\pm(X_{i,m})-f_\pm(X_i)\|_1\longrightarrow0.
\end{align*}
Lieb's theorem involves the logarithms of these MGFs, so we also need their convergence. For $C,D\succeq I$ with trace-class difference, the integral representation of $\log$ and the resolvent identity give
\begin{align*}
 \log D-\log C
 &=\int_0^\infty(C+tI)^{-1}(D-C)(D+tI)^{-1}\,dt,\\
 \|\log D-\log C\|_1
 &\leq\|D-C\|_1\int_0^\infty(1+t)^{-2}\,dt
 =\|D-C\|_1.
\end{align*}
Here the resolvent norms are at most $(1+t)^{-1}$ because $C,D\succeq I$. Both $M_{i,m}^\pm$ and $M_i^\pm$ have this lower bound. Consequently, their summed log MGFs, or cumulants, satisfy
\begin{align} \label{eq:finite_rank_cumulant_convergence}
 K_m^\pm:=\sum_{i=1}^n\log M_{i,m}^\pm
 \longrightarrow K^\pm:=\sum_{i=1}^n\log M_i^\pm
 \quad\text{in trace norm}.
\end{align}

\textit{Details of step~\ref{step:finite_rank_lieb}.}
We now apply Lieb's concavity theorem (Fact~\ref{fact:liebs_concavity_theorem}) in finite dimension. Independence allows us to condition on one increment at a time and replace its exponential by its MGF, giving $\E\tr e^{\pm\theta S_{n,m}}\leq\tr e^{K_m^\pm}$ on $P_m\hs$. Subtracting the finite trace of the identity on this subspace and using $\E S_{n,m}=0$ yields
\begin{align} \label{eq:finite_rank_lieb_bound}
 \E\tr\phi(\pm\theta S_{n,m})
 \leq\tr(e^{K_m^\pm}-I).
\end{align}
To pass the right-hand side of~\eqref{eq:finite_rank_lieb_bound} to the limit, the exponential power series gives
\begin{align*}
 \|(e^{K_m^\pm}-I)-(e^{K^\pm}-I)\|_1
 \leq e^{\max\{\|K_m^\pm\|,\|K^\pm\|\}}
       \|K_m^\pm-K^\pm\|_1\longrightarrow0.
\end{align*}
The exponential factor is bounded because the cumulants converge in trace norm by~\eqref{eq:finite_rank_cumulant_convergence}, so the traces on the right-hand side of~\eqref{eq:finite_rank_lieb_bound} converge. On the left-hand side, the almost sure trace convergence following from~\eqref{eq:finite_rank_trace_convergence} and nonnegativity permit Fatou's lemma. Combining these two limiting arguments with~\eqref{eq:finite_rank_lieb_bound}, we obtain
\begin{align} \label{eq:finite_rank_trace_mgf_bound}
 \E\tr\phi(\pm\theta S_n)
 &\leq\liminf_{m\to\infty}\E\tr\phi(\pm\theta S_{n,m})\nonumber\\
 &\leq\lim_{m\to\infty}\tr(e^{K_m^\pm}-I)
 =\tr(e^{K^\pm}-I)\nonumber\\
 &\leq\tr\big(e^{\psi(|\theta|)V_n}-I\big)<\infty.
\end{align}
Here the first inequality is Fatou's lemma, the second is the finite-dimensional Lieb bound~\eqref{eq:finite_rank_lieb_bound}, and the equality follows from the trace-norm convergence of the centered exponentials. These relations establish $\E\tr\phi(\pm\theta S_n)\leq\tr(e^{K^\pm}-I)$ on $\hs$. This is the centered trace bound obtained from Lieb's concavity, now passed from the finite-dimensional subspaces to the Hilbert space.

The final inequality in~\eqref{eq:finite_rank_trace_mgf_bound} then uses the MGF assumption~\eqref{eq:mgf_assumption} on the original operators: $0\preceq K^\pm\preceq\psi(|\theta|)V_n$. By the min-max characterization, their eigenvalues, ordered from largest to smallest, satisfy $\lambda_j(K^\pm)\leq\psi(|\theta|)\lambda_j(V_n)$. The centered exponentials are positive and trace-class by step~\ref{step:finite_rank_traces}, so their traces are the sums of $e^{\lambda_j(K^\pm)}-1$ and $e^{\psi(|\theta|)\lambda_j(V_n)}-1$, respectively. Scalar monotonicity compares each pair of summands and hence their sums. Averaging the two signs gives the same bound for $\E\tr\varphi(\theta S_n)$, and repeating the argument with the first $k$ increments gives finite expectations for every $k\leq n$.

\textit{Details of step~\ref{step:finite_rank_submartingale}.}
Finally, we establish the submartingale property required by Doob's inequality. Here we only need $\E\tr f(\theta S_n)<\infty$ for $f=\phi$ or $f=\varphi$, as supplied either by~\eqref{eq:main_upper_bound} or by~\eqref{eq:finite_rank_trace_mgf_bound}. Thus this step also applies to the compact summands of Theorem~\ref{theorem:master_theorem}, without the MGF assumption. Let $\F_k=\sigma(X_1,\ldots,X_k)$, $T_{k,m}=\tr f(\theta S_{k,m})$, and $T_k=\tr f(\theta S_k)$.

We first justify convergence of these traces even when the summands are only assumed compact. For compact self-adjoint $A$, $\|A_m-A\|\to0$, so continuity of the functional calculus gives $f(\theta A_m)\to f(\theta A)$ in operator norm. Because $f(0)=0$, the matrix trace of $f(\theta A_m)$ agrees with its trace on $\hs$. Fatou's lemma for the nonnegative diagonal entries, together with~\eqref{eq:finite_rank_trace_compression}, gives
\begin{align*}
 \tr f(\theta A)
 \leq\liminf_{m\to\infty}\tr f(\theta A_m)
 \leq\limsup_{m\to\infty}\tr f(\theta A_m)
 \leq\tr f(\theta A),
\end{align*}
with extended traces allowed at this stage. Applying~\eqref{eq:finite_rank_trace_compression} to the nested subspaces $P_m\hs\subset P_{m+1}\hs$ also gives monotonicity. Consequently,
\begin{align} \label{eq:finite_rank_monotone_traces}
 T_{k,m}\uparrow T_k\quad\text{almost surely}.
\end{align}

To establish integrability at all times from the terminal bound, conditional trace Jensen (Fact~\ref{fact:operator_trace_jensen}) and $\E[S_{n,m}\mid\F_k]=S_{k,m}$ give
\begin{align} \label{eq:finite_rank_terminal_integrability}
 0\leq T_{k,m}\leq\E[T_{n,m}\mid\F_k]\leq\E[T_n\mid\F_k].
\end{align}
Taking expectations and using monotone convergence in~\eqref{eq:finite_rank_monotone_traces} yields $\E T_k\leq\E T_n<\infty$. All the compressed traces are integrable as well. Conditional trace Jensen (Fact~\ref{fact:operator_trace_jensen}) and $\E[X_{k+1,m}\mid\F_k]=0$ now give, as in Lemma~\ref{lemma:submartingale},
\begin{align} \label{eq:finite_rank_conditional_jensen}
 \E[T_{k+1,m}\mid\F_k]
 \geq\tr f\big(\theta\E[S_{k+1,m}\mid\F_k]\big)
 =T_{k,m}.
\end{align}
To pass~\eqref{eq:finite_rank_conditional_jensen} to the full space, conditional monotone convergence and~\eqref{eq:finite_rank_monotone_traces} yield
\begin{align*}
 \E[T_{k+1}\mid\F_k]
 =\lim_{m\to\infty}\E[T_{k+1,m}\mid\F_k]
 \geq\lim_{m\to\infty}T_{k,m}=T_k.
\end{align*}
Since $T_k$ is nonnegative and integrable by~\eqref{eq:finite_rank_terminal_integrability}, this proves that $(T_k)_{k=0}^n$ is a nonnegative submartingale. We have therefore established both ingredients needed by the master theorems: the trace MGF bound and the validity of Doob's maximal inequality in the full space.

\end{document}